\documentclass[12pt]{amsart}
\usepackage{amscd}
\usepackage{amssymb}

\numberwithin{equation}{section}

\swapnumbers

\theoremstyle{plain} 
\newtheorem{thm}[equation]{Theorem}

\newtheorem{lem}[equation]{Lemma}
\newtheorem{prop}[equation]{Proposition}

\theoremstyle{definition}
\newtheorem{defn}[equation]{Definition}

\theoremstyle{remark}
\newtheorem{rem}[equation]{Remark}
\newtheorem{ex}[equation]{Example}

\newtheorem*{VariableNoNum}{{\VariableText}}
\newtheorem{Variable}[equation]{{\VariableText}}

\theoremstyle{definition}
\newtheorem*{VariableNoNumBold}{{\VariableText}}
\newtheorem{VariableBold}[equation]{{\VariableText}}

\newenvironment{titled}[1]
     {\def\VariableText{{#1}}\begin{VariableNoNum}}
     {\end{VariableNoNum}}

\newenvironment{numbered}[1]
     {\def\VariableText{{#1}}\begin{Variable}}
     {\end{Variable}}

\newenvironment{SubSection}[1]
     {\def\VariableText{{\textrm{\textbf{#1}}}}\begin{VariableNoNumBold}}
     {\end{VariableNoNumBold}}

\newenvironment{NumberedSubSection}[1]
     {\def\VariableText{{\textrm{\textbf{#1}}}}\begin{VariableBold}}
     {\end{VariableBold}}

\newlength{\asidelength}

\def\Changed/{\ifvmode\else\vadjust{%
\vbox to 0pt{\vskip -\baselineskip%
\hbox to 0pt{\hss\vrule height 0pt depth 1.2\baselineskip\hskip 1em}\vss}}\fi}
\def\CHanged{\ifvmode\else\vadjust{%
\vbox to 0pt{\vskip -\baselineskip%
\hbox to 0pt{\hss\vrule height 0pt depth 1.2\baselineskip\hskip 1em}\vss}}\fi}

\def\Math#1{\def\MathString{#1}\futurelet\MathDelim\MathChoose}
\def\MathChoose{\ifmmode\let\MathDo\MathString%
              \else\let\MathDo\MathSkip\fi%
              \MathDo}
\def\MathSkip{\ifx\MathDelim/\def\MathDo{$\MathString$\EatOne}%
              \else\def\MathDo{$\MathString$}\fi%
              \MathDo}

\def\Text#1{\def\TextString{#1}\futurelet\TextDelim\TextSkip}
\def\TextSkip{\ifx\TextDelim/\def\TextDo{\TextString\EatOne}%
              \else\let\TextDo\TextString\fi%
              \TextDo}
\def\EatOne#1{}

\def\SkipToEndScan#1\EndScan{}
\def\Scan#1#2#3{\ifx#1#2#3\expandafter\SkipToEndScan\fi\Scan#1}
\def\Upper#1{%
\Scan#1aAbBcCdDeEfFgGhHiIjJkKlLmMnNoOpPqQrRsStTuUvVwWxXyYzZ#1#1\EndScan}
\def\Phrase#1 #2/#3/#4=#5 #6/#7/#8.{%
\expandafter\edef\csname#2#3\endcsname{\noexpand\Text{#6#7}}
\expandafter\edef\csname\Upper#2#3\endcsname{\noexpand\Text{\Upper#6#7}}
\expandafter\edef\csname#1#2#3\endcsname{\noexpand\Text{#5 #6#7}}
\expandafter\edef\csname\Upper#1#2#3\endcsname{\noexpand\Text{\Upper#5 #6#7}}
\expandafter\edef\csname#2#4\endcsname{\noexpand\Text{#6#8}}
\expandafter\edef\csname\Upper#2#4\endcsname{\noexpand\Text{\Upper#6#8}}
}

\Phrase a discrete//s=an ordinary//s.

\newcommand{\whatever}{\text{--}}

\newcommand{\Z}{\Math{\mathbb{Z}}}
\newcommand{\F}{\Math{\mathbb{F}}}
\newcommand{\Fp}{\Math{\F_p}}
\newcommand{\Fpn}{\Math{\F_{p^n}}}
\newcommand{\Zp}{\Math{\Z_p}}
\newcommand{\Zpinfty}{\Math{\Z/p^\infty}}
\newcommand{\Q}{\Math{\mathbb{Q}}}

\newcommand{\Tensor}{\otimes}

\newcommand{\RightArrow}[1]{\xrightarrow{#1}} 

\newcommand{\Hom}{\operatorname{Hom}}
\newcommand{\Aut}{\operatorname{Aut}}
\newcommand{\Map}{\operatorname{Map}}
\newcommand{\End}{\operatorname{End}}

\newcommand{\hhh}{\operatorname{h}\!}
\newcommand{\thhh}{\tilde{\operatorname{h}}}
\def\HomotopyOrbit#1on#2/{\ensuremath{#2_{\hhh#1}}}
\def\RedHomotopyOrbit#1on#2/{\ensuremath{#2_{\thhh#1}}}

\newcommand{\weq}{\sim}
\newcommand{\iso}{\cong}

\newcommand{\Sphere}{\Math{\mathbb{S}}}
\newcommand{\Sph}{\Sphere}
\newcommand{\SP}{\Sphere}
\newcommand{\II}{\Math{\mathbb{I}}}

\newcommand{\IGeneric}{\Math{\mathcal{I}}}
\newcommand{\IIR}{\Math{\IGeneric}}
\newcommand{\IIRhat}{\Math{\hat{\IIR}}}

\newcommand{\IIRG}{\Math{\mathcal{G}}}

\newcommand{\IIRO}{\Math{\IGeneric}}
\newcommand{\IIpG}{\IIRG}
\newcommand{\IIpO}{\IIRO}

\newcommand\Mod{\operatorname{Mod}}
\newcommand{\RMod}{\Math{{}_R\Mod}}

\newcommand{\HomR}{\Hom_R}
\newcommand{\Ext}{\operatorname{Ext}}
\newcommand{\Tor}{\operatorname{Tor}}
\newcommand{\Cell}{\operatorname{Cell}}
\newcommand{\Cellk}{\operatorname{Cell}_k}

\Phrase an salgebra//s=an \relax\relax$\noexpand\Sphere$-algebra//s.
\Phrase an ec//s=an effectively constructible//s.
\Phrase a kequivalence//s=a \relax\relax$\noexpand\Cellk$-equivalence//s.

\newcommand{\TensorR}{\otimes_R}
\newcommand{\EndR}{\operatorname{End}_R}

\newcommand{\Lof}{\Math{L}}

\newcommand{\lbbracket}{[\![}
\newcommand{\rbbracket}{]\!]}

\newcommand{\KSmash}{\hat\Tensor}
\newcommand{\hatSmash}{\KSmash}
\newcommand{\SK}{\Math{\hat{\SPP}}}

\newcommand{\EHS}{\Math{\epsilon}}  

\newcommand{\KTensor}{\KSmash}

\newcommand{\Sn}{\Math{\mathcal{S}}}  
\newcommand{\SKn}{\Math{\hat{\Sn}}}   

\newcommand{\En}{\Math{{E}}}  
\newcommand{\kn}{\Math{{K}}}  

\newcommand{\Kp}{\kn}
\newcommand{\Ep}{\En}
\newcommand{\Sphp}{\Sn}
\newcommand{\SPP}{\Sn}

\newcommand{\KPP}{\kn}

\newcommand{\EPP}{\En}

\newcommand{\EC}{\Math{\mathcal{E}}}

\newcommand{\EndEn}{\EC}

\newcommand{\EndE}{\Math{\EndEn}}

\newcommand{\TensorEC}{\otimes_{\EC}}

\newcommand{\Eph}{\Math{\En^{\vee}}}

\newcommand{\Dual}[1]{D_{#1}}
\newcommand{\DK}{\hat D}

\newcommand{\Qp}{\Math{\mathbb{Q}_p}}
\newcommand{\Tensord}{\oslash}

\newcommand{\Line}{\Math{\mathcal{L}}}
\newcommand{\Lined}{\Math{\mathcal{L}^\#}}

\newcommand{\RR}{\Math{S}}

\newcommand{\mfm}{\Math{\mathfrak{m}}}

\newcommand{\mfn}{\Math{\mathfrak{n}}}

\newcommand{\Typen}{\Math{\mathcal{F}}}
\newcommand{\ModPic}{\Math{\mathbf{Pic}}}

\Phrase a coinduction//s=a coinduction//s.
\Phrase a coinduced//s=a coinduced//s.

\newcommand{\Rt}{\Math{R_\epsilon}}

\Phrase a finitelength//s=a finite length//s.
\begin{document}

\title[Gross-Hopkins duality]{Gross-Hopkins duality\\ and the Gorenstein condition}
\author {W. G. Dwyer, J. P. C. Greenlees and S. B. Iyengar}

\address{Department of Mathematics, University of Notre Dame, Notre Dame, Indiana 46556, USA} 
\email{dwyer.1@nd.edu}
\address{Department of Pure Mathematics, Hicks Building, Sheffield S3 7RH, UK}
\email{j.greenlees@sheffield.ac.uk}
\address{Department of Mathematics, University of Nebraska, Lincoln, NE 68588, USA}
\email{iyengar@math.unl.edu}
\date{27th May 2009}

\thanks {JPCG acknowledges EPSRC support. SBI was partly supported by NSF grant DMS 0602498}

  \begin{abstract}
Gross and Hopkins have proved that in chromatic stable homotopy,  Spanier-Whitehead duality  
nearly coincides with Brown-Comenetz duality. Our goal is to give a conceptual interpretation 
for this phenomenon in terms of the Gorenstein condition for maps of ring spectra in the sense
of \cite{rDGI}
We describe a general notion of Brown-Comenetz dualizing module for a map of ring spectra
and show that in this context such dualizing modules correspond bijectively to invertible $K(n)$-local spectra.
\end{abstract}


\maketitle


\section{Introduction}\label{SIntroduction}
Suppose that \SP/ is the sphere spectrum and \II/ its Brown-Comenetz
dual. The Spanier-Whitehead dual $\Dual\Sphere X$ of a spectrum $X$
is defined to be the mapping spectrum $\Map(X,\SP)$, while the
Brown-Comenetz dual $\Dual\II X$ is the spectrum $\Map(X,\II)$.  
These are very different from one another: for instance, Spanier-Whitehead duality
behaves well on homology (if $X$ is finite then
$H_i(\Dual\Sphere X)\iso H^{-i}X$), while Brown-Comenetz duality
behaves well on homotopy ($\pi_i(\Dual\II
X)\iso\Hom(\pi_{-i}X,\Q/\Z)$).

Nevertheless, Gross and Hopkins \cite{rHopkinsGross} have proved that
in some localized stable homotopy situations, the appropriate version
of Spanier-Whitehead
duality  nearly coincides with
Brown-Comenetz duality. Our goal is to give 
a conceptual interpretation for this phenomenon. To do this we describe the notion
of a \emph{Brown-Comenetz dualizing module} $\IIR$
(\ref{DBrownComenetz}) for a ring spectrum map $R\to k$. 
There are two common mechanisms for the construction of a
dualizing module:
\begin{enumerate}
\item \label{GorCase} if $R\to k$ is Gorenstein
  (\ref{DefineGorenstein}),  there is a way to obtain a dualizing
  module $\IIRG$ from $R$~itself. The duality functor
 $\Dual\IIRG$ over~$R$ agrees with Spanier-Whitehead duality over~$R$
 (\ref{DualitiesCoincide});
\item if $R$ satisfies a  different (milder) condition, there is a``trivial''
  dualizing module $\IIR_0$ constructed by coinduction
  (\ref{InduceItUp}) from \II/ and the unit map $\SP\to R$. The
  duality functor $\Dual{\IIR_0}$ over~$R$ then agrees with garden-variety
    Brown-Comenetz duality over~\SP.
\end{enumerate}

When both of the above constructions go through, the question of whether
$\IIRG\sim\IIR_0$, or equivalently of whether
$\Dual\IIRG\sim\Dual{\IIR_0}$, is a
type of orientability issue (\ref{PDDuality}).
In the basic case of the sphere spectrum, only (2) applies, and so
there is no call to measure Brown-Comenetz duality against
Spanier-Whitehead duality by comparing the kind of Brown-Comenetz
dualizing modules we consider.
Things are different in the context of Gross--Hopkins duality. In this
case \emph{both} $\IIRG$ and $\IIR_0$ exist, and although they do not
quite agree, they are very similar to one another. 
We indicate in \ref{ComplicatedNaturality}
 why $\IIRG$ can be distinguished from $\IIR_0$ by the algebraic calculation made
in~\cite{rGHVectorBundles}, and interpret this as closely analogous
to distinguishing two spherical fibrations by calculating their
Stiefel-Whitney classes.  We classify all of the
Brown--Comenetz dualizing modules in this case (\ref{GHClassifyBC}),
and explain why they correspond bijectively to invertible modules.

In describing our point of view, we start with the general notion
of Brown-Comenetz duality and use this to describe the homotopical form
of Gorenstein duality \cite{rDGI}. To put these ideas in a more familiar
homotopy theoretic context, we point out that 
Poincar\'e duality is a special case of Gorenstein duality.  Finally
we indicate how Gross-Hopkins duality fits into this framework.
This paper could not have been written without \cite{rHS} and
\cite{rS}; a lot of what we do is to give a different
slant to the material in \cite{rS}. Although our treatment has an
intrinsic interest, it can also be viewed as an extended example of the
theory of \cite{rDGI}, an example which highlights the importance of
orientability issues.

\begin{numbered}{Some notation}\label{FiniteLength}
  We refer to a ring spectrum $R$ as an \emph{\salgebra/}, and a
  module spectrum over $R$ as an \emph{$R$-module} \cite{rEKMM} \cite{rHSS}. A map between
  spectra is an \emph{weak equivalence} (\emph{equivalence} for short)
  if it induces an isomorphism on homotopy groups. If $M$, $N$ are
  left $R$-modules, then $\HomR(M,N)$ denotes the spectrum of
  (derived) $R$-module maps between them; if $M$ is a left $R$-module
  and $N$ a right $R$-module, then $N\TensorR M$ is the (derived)
  smash product of $M$ and $N$ over $R$.

  There's no harm in treating an ordinary ring $R$ as \ansalgebra.  In that case a \emph{module}
  over $R$ in our sense corresponds to what is usually called a chain
  complex over $R$, $\HomR(M,N)$ to the derived mapping complex, and
  $N\TensorR M$ to the derived tensor product.  If $R$ is an ordinary 
  ring and $M$, $N$ are \discrete/  left $R$-modules,
  treated as chain complexes concentrated in degree~$0$, then
  according to our conventions $\HomR(M,N)$ is a spectrum with
  $\pi_i\HomR(M,N)\iso\Ext^{-i}_R(M,N)$. Similarly, if $N$ is 
  \adiscrete/ right $R$-module, then $N\TensorR M$ is a spectrum with
  $\pi_i(N\TensorR M)\iso \Tor_{i}^R(N,M)$. In these  cases
  we write $\Ext^0_R(M,N)$ for the usual group of homomorphisms $M\to
  N$, and $N\Tensord_RM=\Tor_0^R(N,M)$ for the usual tensor product.

  If $R$ is \adiscrete/ ring with a distinguished maximal
  ideal~$\mfm$, we will refer to \adiscrete/ finitely generated
  \mfm-primary torsion $R$-module as a \emph{finite length $R$-module}.

  If $R$ is \ansalgebra/ and $k$, $M$ are $R$-modules, then
  $\Cellk(M)$ denotes the \emph{$k$-cellular approximation} of $M$:
  $\Cellk(M)$ is built from $k$ (\ref{DDefineCellular}), and there is a map $\Cellk(M)\to M$
  which is a \emph{$\Cellk$-equivalence}, i.e., induces an equivalence
  on $\Hom_R(k,\whatever)$.  
  
\end{numbered}

\begin{NumberedSubSection}{Brown-Comenetz duality}\label{DefineEC}
  Suppose that $R\to k$ is a map of \salgebras. Let $\EC$ be the derived
  endomorphism \salgebra/ $\EndR(k)$.  An $R$-module $M$ is said to be
  \emph{\ec/ from $k$} if the natural evaluation map
  \begin{equation}\label{EEffectively}
          \HomR(k,M)\TensorEC k\to M
  \end{equation}
  is an equivalence (cf. \ref{ConstructingCellk}). 

  \begin{rem}
    If $M$ is \ec/ from $k$ then $M$ is built from $k$  as an
    $R$-module.
    For some $R$ and $k$, the converse holds (\ref{Proxy}). 
  \end{rem}

   \begin{defn}
   \label{DBrownComenetz}
    A \emph{Brown-Comenetz dualizing module} for $R\to k$ is an
    $R$-module $\IIR$ which is effectively constructible from $k$ and
    has the property that, for some $d\ge0$, $\HomR(k,\IIR)$ is
    equivalent as a left $k$-module to $\Sigma^dk$.
  \end{defn}
  
  Giving such a dualizing module \IIR/ involves finding a way of extending to
  $R$-modules the notion of ordinary (i.e., Spanier-Whitehead) duality
  for $k$-modules. As \ref{EEffectively} suggests, in favorable cases
  \cite[6.9]{rDGI} these dualizing modules correspond to appropriate
  right $\EC$-module structures on (a suspension of)~$k$.

  \begin{numbered}{Examples (uniqueness)} \label{BCAlgebraic} \cite[\S5]{rDGI}
    The module $\Z/p^\infty$ is a Brown-Comenetz dualizing module for
    $\Z\to \Z/p$. The $p$-primary summand of the spectrum $\II$ is a
    Brown-Comenetz dualizing module for $\Sph\to\Z/p$.  
    In both of these cases, up to suspension and equivalence there
    is only one Brown-Comenetz dualizing module for $R\to k$.
  \end{numbered}

  \begin{numbered}{Examples (non--uniqueness)}\label{ExNonUnique}
    Let $X$ be a $1$-connected based finite CW-complex.  Let
    $k$ denote $\Sph$, and let $R=C^*(X;k)$ denote the
    Spanier-Whitehead dual (over \Sph/) of the unreduced suspension
    spectrum of $X$.  Then $R$ is \ansalgebra/ under a
    multiplication induced by the diagonal map, and there is an
    augmentation $R\to k$ given by restriction to be basepoint of $X$.
    As in \ref{PDDuality}, Brown--Comenetz dualizing modules for $R\to
    k$ correspond bijectively up to equivalence to stable spherical
    fibrations over~$X$.
  \end{numbered}

  \begin{numbered}{Examples (\Coinduction/)}\label{InduceItUp}
    Suppose that $T\to R$ is a map of \salgebras, and that $J$ is a
    Brown-Comenetz dualizing module for $T\to k$.  Let
    $\IIR=\Cell^R_k\Hom_T(R,J)$.  If $\IIR$ is effectively
    constructible from $k$, then $\IIR$ is a Brown-Comenetz dualizing
    module for $R\to k$, called the Brown-Comenetz dualizing module
    \emph{\coinduced/} from $J$.
  \end{numbered}

\end{NumberedSubSection}
    
\begin{NumberedSubSection}{Gorenstein duality}
  Let $f:R\to k$ be as above.

  \begin{defn}\label{DefineGorenstein} \cite[8.1]{rDGI}
    The map $f:R\to k$ is \emph{Gorenstein} if $\Cellk(R)$ is a
    Brown-Comenetz dualizing module for $f$. 
  \end{defn}

  \begin{rem}\label{DualitiesCoincide}
    Suppose $R\to k$ is Gorenstein, with associated Brown-Comenetz dualizing module $\IIRG=\Cellk(R)$. 
    The map $\IIRG\to R$ induces an equivalence $\HomR(M,\IIRG)\to\HomR(M,R)$  for $M=k$ and thus for any $R$-module $M$  built from $k$. For such $M$, this gives an equivalence
    \[
                \Dual\IIRG M\weq \Dual RM\,.
    \]
    In other words, if $R\to k$ is Gorenstein, then for $R$-modules
    which are built from $k$, Spanier-Whitehead duality agrees 
    with the  variant of Brown-Comenetz duality singled out
    by the Gorenstein condition.
  \end{rem}

  \begin{numbered}{Example (algebra)}\label{GorAlgebraExamples} 
    Suppose that $R$ is the formal power series ring
    $\Zp\lbbracket x_1,\ldots,x_{n-1}\rbbracket$, that $\mfm
    \subset R$ is its maximal ideal, and that $k\iso R/\mfm$
    is its residue field $\Fp$.  The map $R\to k$ is Gorenstein, with
    associated Brown-Comenetz dualizing module $\IIRG$.  For 
    \afinitelength/ (\ref{FiniteLength}) $R$-module~$M$,
    the dual $\Dual\IIRG(M)$ is given by
    \[
      \Dual{\IIRG}(M)\weq \Sigma^{-n}\Ext^n_R(M,R)\,.
    \]
    Precisely as in \ref{BCAlgebraic}, the $\Zp$-module $\Zpinfty$
    is a Brown-Comenetz dualizing module for $\Zp\to\Fp$, and as in
    \ref{InduceItUp} there is \acoinduced/ Brown-Comenetz dualizing
    module $\IIR=\IIR_0$ for $R\to k$. For $M$ as before the dual
    \[
      \Dual{\IIR}(M)\weq\Ext^0_{\Zp}(M,\Zpinfty) 
    \]
    is the ordinary Pontriagin dual of $M$.  It turns out
    (\ref{SFullDuality}) that $\IIRG$
    is equivalent as an $R$-module to $\Sigma^{-n}\IIR$, and hence
    that on the category of \finitelength/  $R$-modules, the functor
    $\Ext^n_R(\whatever,R)$ is naturally isomorphic to
    $\Ext^0_{\Zp}(\whatever,\Zpinfty)$.
  \end{numbered}

  \begin{numbered}{Example (Poincar\'e duality)}\label{PDDuality}\label{ListPoincareBC}
    (See \cite{rDGI} and \cite{rKlein}.)  This example is based on the
    following theorem.
    \begin{prop}\label{SWisGor}
      Suppose that $X$ is a based finite $1$-connected CW-complex,
      $k=\Sph$, and $R=C^*(X;k)$, as in \ref{ExNonUnique}. Then $R\to
      k$ is Gorenstein if and only if $X$ is a Poincar\'e duality
      space.
    \end{prop}

      In the situation of \ref{SWisGor}, there are usually many
      Brown-Comenetz dualizing modules for $R\to k$: these are exactly
      the Thom spectra $X^\rho$ obtained from stable spherical
      fibrations $\rho$ over $X$. 
      If $X$ is a Poincar\'e duality space of formal
      dimension~$d$, then as in   \cite{rAtiyah}  the Brown-Comenetz dualizing module
      $\IIRG=\Cellk(R)\weq R$ provided by the Gorenstein condition
      \cite[8.6]{rDGI} is
      $X^\nu$, where $\nu$ is the stable Spivak normal bundle of $X$,
      desuspended to have stable fibre dimension $-d$. 

      Since the spectrum $k$ is a Brown-Comenetz dualizing module for $k\to
      k$, it follows as in \ref{InduceItUp} there is \acoinduced/
      Brown-Comenetz dualizing module $\IIR=\IIR_0$ for $R\to k$
      (\cite[9.16]{rDGI}, \ref{RecognizeGorenstein}). This
      \coinduced/ dualizing module is the Thom complex $X^0$ of the
      trivial bundle.
      The $R$-module $\IIRG$ is equivalent to $\IIR$ (up to
      suspension) if and only if $\nu$ is trivial, or in other words
      if and only if $X$ is orientable for stable cohomotopy.

      Observe that by the Thom isomorphism theorem, the  two
      dualizing modules $\IIRG$ and $\IIR$  cannot be distinguished
      by mod~$2$ cohomology, although they can sometimes be
      distinguished by the action of the Steenrod algebra on mod~$2$
      cohomology.
\end{numbered}

  \begin{numbered}{Aside on functoriality}\label{LFunctoriality}
    For later purposes we describe an extended functoriality property
    of the isomorphisms described in \ref{GorAlgebraExamples}.  Let
    $R\to k$ be as in \ref{GorAlgebraExamples}, but widen the module
    horizon to include the category of \finitelength/ (\ref{FiniteLength}) \emph{skew} $R$-modules: the objects
    are \discrete/ $R$-modules as before, but a map $M\to M'$ is a pair
    $(\sigma,\tau)$, where $\sigma$ is an automorphism of $R$ and
    $\tau:M\to M'$ is a map of abelian groups such that for $r\in R$ and
    $m\in M$, $\tau(rm)=\sigma(r)\tau(m)$. Both $\Dual{\IIRG}$ and
    $\Dual{\IIR}$ extend to this larger category (with the same
    definitions as before), but the functors are \emph{not} naturally
    equivalent there. This is reflected in the fact that if
    $G=\Aut(R)$, then the twisted group ring $R[G]$ acts naturally
    both on $\IIRG$ and on $\IIR$ in such a way that $\IIRG$ and $\IIR$ are equivalent as
    $R$-modules, but not as $R[ G]$-modules.  The discrepancy between
    \IIRG/ and \IIRO/ has a simple description. Let
    $S=\Zp\lbbracket x_1,\ldots,x_{n-1},y_1\ldots,y_{n-1}\rbbracket$ be the evident
    completion of $R\Tensord_{\Zp}R$ and let
    $\Line=\Tor_{n-1}^{S}(R,R)$.  (The object $\Line$ might be
    characterized as
    a type of Hochschild homology group
    of $R$.) Then $\Line$ is \adiscrete/ $R[G]$-module which is free of
    rank~$1$ as an $R$-module, and there is a natural map
    $\Sigma^n\Line\Tensor_R\IIRG\to\IIR$ of $R[G]$-modules which is an
    equivalence (\ref{SFullDuality}). (The action of $G$ on the tensor product is
    diagonal). This implies that on the category of \finitelength/ skew $R$-modules
    there is a natural isomorphism of functors
    \[\Line\Tensord_R\Ext_n^R(\whatever,R)\weq\Ext_0^{\Zp}(\whatever,\Zpinfty)\,.\]
  \end{numbered}
\end{NumberedSubSection}

\begin{NumberedSubSection}{Gross-Hopkins duality}\label{GHIntro}
  Fix an integer $n\ge1$, and let $\Lof=L_n$ denote the localization
  functor on the stable category corresponding to the homology theory
  $K(n)\vee\cdots\vee K(0)$ , where $K(i)$ is the $i$'th Morava
  $K$-theory.  Let $\Sphp=L_n(\Sphere)$ and let $\Kp=K(n)$.  There is
  an essentially unique \salgebra/ homomorphism $\Sphp\to \Kp$.  The
  first component of Gross-Hopkins duality is the following
  statement.

  \begin{thm}\label{GHShift}
    The homomorphism $\Sphp\to\Kp$ is Gorenstein.
  \end{thm}

  This theorem provides a Brown-Comenetz dualizing module
  $\IIpG=\Cellk\Sphp$ for $\Sphp\to\Kp$.  The ordinary Brown-Comenetz
  dualizing spectrum \II/ is a Brown-Comenetz dualizing module for
  $\Sphere\to\Kp$; as in \ref{InduceItUp} this gives rise to \acoinduced/
  Brown-Comenetz dualizing module
  $\IIpO=\Cellk\Hom_{\Sphere}(\Sphp,\II)$ for $\Sphp\to\Kp$
  (\ref{RecognizeGorenstein}, \ref{knIsKoszul}).  The
  second component of Gross-Hopkins duality is the assertion that
  \IIpG/ \emph{cannot} be distinguised from \IIpO/ by the most
  relevant applicable homological functor. This is analogous in this
  context to the Thom isomorphism theorem (cf. \ref{ListPoincareBC}).

  Let $\Ep$ be the \salgebra/ of \cite{rS}, with
  \[
       \Ep_*= \pi_*\Ep = W \lbbracket u_1,\ldots,u_{n-1}\rbbracket[u,u^{-1}]\,,
  \]
  where $u_k$ is of degree~$0$, $u$ is of degree~$2$, and $W$ is the
  Witt ring of the finite field $\F_{p^n}$. For spectra $X$ and $Y$,
  let $X\hatSmash Y=L_{\Kp}(X\Tensor Y)$,   where $L_{\Kp}$ is
  localization with respect to $\Kp$. Following \cite{rS}, for
  any $X$ we write $\Eph_*(X) =\pi_*(\Ep\hatSmash X)$.

  \begin{thm}\label{GHSame}
    Both $\Eph_*\IIpG$ and $\Eph_*\IIpO$ are rank~$1$ free modules
    over $\Ep_*$.
  \end{thm}

  The final and most difficult component of Gross-Hopkins duality is a
  determination of how $\Eph_*\IIpG$ differs from $\Eph_*\IIpO$ as a
  module over the ring of operations in $\Ep_*$; this is analogous to
  distinguishing between two Thom complexes by considering the action
  of Steenrod algebra on mod~$2$ homology (cf.  \ref{ListPoincareBC}).

  We begin by comparing the homologies of $\Dual\IIpG(\Typen)$ and
  $\Dual\IIpO(\Typen)$ when $\Typen$ is a finite complex of type~$n$, i.e.,
  a  module over $\Sphp$ which is finitely built from $\Sphp$ and has
  $K(i)_*\Typen=0$ for $i<n$ and $K(n)_*\Typen\ne0$. 
  These conditions imply that each
  element of $\Eph_*\Typen$ is annihilated by some power of the maximal
  ideal $\mathfrak m\subset \Ep_0$ \cite[8.5]{rHS}.

  \begin{prop}\label{SimpleNaturality}
    Suppose that $\Typen$ is a finite complex of type~$n$. Then there are
    natural isomorphisms
    \[
        \begin{aligned}
        \Eph_{-i}\Dual\IIpG \Typen&\,\iso\, \Ext^n_{\Ep_0}(\Eph_{i-n}\Typen, \Ep_0)\\
        \Eph_{-i}\Dual\IIpO \Typen &\,\iso\, \Ext^0_{\Zp} (\Eph_{i+n^2}\Typen, \Zpinfty)\,.
        \end{aligned}
    \]
  \end{prop}

  Recall \cite{rS} that the Morava stabilizer group $\Gamma$, in one
  of its forms, is a profinite group of multiplicative automorphisms
  of $\Ep$.  The ring $\pi_*\End_{\Sphp}(\Ep)$ is the completed
twisted group ring $\Ep_*\lbbracket\Gamma\rbbracket$ (see \cite[pf. of
Prop.~16]{rS}), and so, up to completion and multiplication by
elements in $\Ep_*$, the operations in $\Ep_*$ are all of degree~$0$
and are determined by the action of elements of~$\Gamma$. If $X$ is a
spectrum, then $\Gamma$ acts on $\Eph_*(X)$ as a group of
automorphisms in the category of skew $\Ep_*$-modules
(\ref{LFunctoriality}). It follows from naturality that the
isomorphisms in \ref{SimpleNaturality} are $\Gamma$-equivariant,
where, for instance, $\Gamma$ acts on $\Ext^n_{\Ep_0}(\Eph_{i-n}\Typen,
\Ep_0)$ in a diagonal way involving actions on all three constituents
of the~$\Ext$. According to \ref{GorAlgebraExamples}, the modules
\[
        \Ext^n_{\Ep_0}(\Eph_i\Typen,\Ep_0) \text{ and }
        \Ext^0_{\Zp}(\Eph_i\Typen,\Ep_0)
\]
are isomorphic for any~$i$; the question is to what extent these
isomorphisms do or do not respect the action of $\Gamma$. 

This is exactly the issue discussed in \ref{LFunctoriality}. Given \ref{SimpleNaturality} and \ref{LFunctoriality}, the following proposition is immediate (cf. \ref{TwistOverO}).
Let 
\[
T=W\lbbracket u_1,\ldots,u_{n-1},u_1',\ldots,u_{n-1}'\rbbracket
\]
be the evident completion of $\Ep_0\Tensord_W\Ep_0$, and let
$\Line=\Tor_{n-1}^T(\Ep_0,\Ep_0)$. 
It turns out (\ref{SFullDuality}) that $\Line$ is a free
module of rank~$1$ over $\Ep_0$.

\begin{prop}\label{ComplicatedNaturality}
  For any finite complex of type~$n$, there are natural isomorphisms
  \[
        \Line\Tensor_{\Ep_0}\Eph_{i-n-n^2}\Dual\IIpG X \iso \Eph_i\Dual\IIpO
        X\,
  \]
  of modules over $\Ep_0\lbbracket\Gamma\rbbracket$.
\end{prop}

\begin{rem}
  We emphasize that in \ref{ComplicatedNaturality} the action of
  $\Gamma$ on the left-hand module is diagonal, and involves a
  nontrivial action of $\Gamma$ on $\Line$.
\end{rem}

  This easily leads to the following proposition.

  \begin{prop}\label{TwistedDifference}
    There are natural isomorphisms of
    $\Ep_0\lbbracket\Gamma\rbbracket$-modules 
    \[
    \Line\Tensor_{\Ep_0}\Eph_{i-n-n^2}\IIpG \iso \Eph_i\IIpO\,.
    \] 
  \end{prop}

As in \cite{rS}, there is a determinant-like map $\det:\Gamma\to\Zp^\times$.  If $M$ is \adiscrete/ module over $\Ep_0\lbbracket\Gamma\rbbracket$ or $\Ep_*\lbbracket\Gamma\rbbracket$, write $M[\det]$ for the module obtained from~$M$ by twisting the action of $\Gamma$ by $\det$. The key computation made in \cite{rGHVectorBundles} by Gross and Hopkins (which we do not rederive) involves the action of
$\Gamma$ on $\Line$.

\begin{thm}\label{GHTheorem}\cite[Th.~6]{rHopkinsGross}
  As a module over the twisted group ring
  $\Ep_0\lbbracket\Gamma\rbbracket$, $\Line$ is isomorphic to $\Ep_{2n}[\det]$.
\end{thm}

The first statement below follows from the fact that $\IIpG\to\Sn$ is
a $\kn_*$-equivalence (\ref{TwoKEquivalences}) and hence an
$\Eph_*$-equivalence; the second is a combination of 
\ref{TwistedDifference} and \ref{GHTheorem}.

  \begin{thm}\label{GHDifferent}\cite{rS}
    There are
    isomorphisms of 
    $\Ep_*\lbbracket\Gamma\rbbracket$-modules:
    \[
       \begin{aligned}
           \Eph_*\IIpG&\,\,\iso\,\, \Ep_* \\
           \Eph_*\IIpO\,&\,\,\iso\,\, \Sigma^{n^2-n}\Ep_*[\det]\,.
       \end{aligned}
    \]
  \end{thm}

Finally, we give an analogue of the classification of Brown-Comenetz
dualizing modules from \ref{ListPoincareBC}. Recall that a
$\KPP$-local spectrum $M$ is said to be \emph{invertible} if there is
a $\KPP$-local spectrum $N$ with $M\hatSmash N\weq L_{\KPP}(\Sphp)$.

\begin{thm}\label{GHClassifyBC}
        Let $\EndE=\End_{\SPP}(\EPP)$.
        Up to equivalence, there are bijective correspondences between
       the following three kinds of objects:
       \begin{enumerate}
         \item invertible $\KPP$-local spectra,
         \item Brown-Comenetz dualizing modules for $\SPP\to \KPP$ \label{BCmodules},
         \item right actions of $\EndE$ on a suspension of $\EPP$ which
           extend the natural right action of $\EPP$ on itself. \label{RightAction}
        \end{enumerate}
\end{thm}

\begin{rem}
  Suppose that $X$ is a based CW-complex, $G$ is the loop space on $X$
  (constructed as a simplicial group), and $\EndE=\Sphere[G]$ is the
  ring spectrum obtained as the unreduced suspension spectrum of $G$.
  Let $k=\Sphere$ and $R=C^*(X;k)$ as in \ref{PDDuality}. Say that a
  module $M$ over $R$ is invertible if there is a module $N$ such that
  $M\Tensor_RN\weq R$. Then if $X$
  is finite and $1$-connected, $\EndE$ is equivalent to $\End_R(k)$
\cite{rDGI}, and \ref{GHClassifyBC} becomes in part analogous to the
statement that up to equivalence there are bijective correspondences
between the following four kinds of objects:
\begin{enumerate}
\item invertible modules over $R$,
\item Brown-Comenetz dualizing modules for $R\to k$,
\item actions of $\EndE$ on a suspension of $k$ which (necessarily) extend the action of
$k$ on itself, and
\item stable spherical fibrations over $X$.
\end{enumerate}
\end{rem}
\end{NumberedSubSection}

\begin{NumberedSubSection}{Organization of the paper}
Section \ref{CCellular} has a short discussion of cellularity,
\S\ref{CCommutative} expands on some of the algebraic issues discussed
in \ref{LFunctoriality}, and \S\ref{CPreliminaries} recalls some
material from stable homotopy theory. Section \ref{SGrossHopkins}
contains the proofs of \ref{GHShift}, \ref{GHSame},
\ref{SimpleNaturality}, and \ref{TwistedDifference}. The last section
has a proof of \ref{GHClassifyBC}.
\end{NumberedSubSection}

\begin{NumberedSubSection}{Notation and terminology}\label{SNotation}
  If $R$ is a ring spectrum, we write ${}_R\Mod$ and
  $\Mod_R$ for the respective categories of left and right module
  spectra over $R$. If $X$ and $Y$ are spectra, $\Hom(X,Y)$ stands for
  $\Hom_{\Sphere}(X,Y)$ and $X\Tensor Y$ for $X\Tensor_{\Sphere}Y$.
  Note, though, that $\Sn\Tensor_{\Sphere}\Sn\weq\Sn$; this implies that
  if $X$ and $Y$ are $\Sn$-modules then $\Hom_{\Sn}(X,Y)\weq\Hom(X,Y)$
  and $X\Tensor_{\Sn}Y\weq X\Tensor Y$. Whenever possible we use the
  simpler notation (without the subscript \Sn). If
  $X$ is a spectrum, $\hat X=L_{\kn}X$ stands for the \kn-localization
  of $X$; we also write $\DK$ for $D_{\SK}$, so that $\DK
  X=\Hom(X,\SK)$. 

  For us, a \emph{finite complex of type~$n$} is an \Sn-module
  \Typen/ which is finitely built from \Sn, such that $K(i)_*\Typen=0$
  for $i<n$ and $K(n)_*\Typen\ne0$. If $F(n)$ is a finite complex of
  type~$n$ in the sense of \cite[\S1.2]{rHS}, then $\Sn\Tensor F(n)$
  is a finite complex of type~$n$ in our sense.
\end{NumberedSubSection}

\begin{SubSection}{Some technicalities}\label{DefineEHS}
  The localized sphere \Sphp/ is a commutative \salgebra, as is the
  spectrum \EPP/  \cite[\S7]{rGoerssHopkins}.
  Disconcertingly, \Kp/ usually has an uncountable number of
  \salgebra/ structures, none of them commutative; nevertheless we
  fix one \salgebra/ structure on \Kp/ (nothing to follow will depend
  on this choice), and work with the essentially unique \salgebra/ map
  $\SPP\to\KPP$. 
\end{SubSection}

\section{Cellularity and Koszul complexes}
\label{CReview}\label{CCellular}

In this section we review the idea of cellularity, and look at how it
fits in with the effective constructibility condition which appears in
the definition of Brown-Comenetz dualizing module.

\begin{NumberedSubSection}{Cellularity and cellular approximation}
Suppose that $R$ is \ansalgebra/ and that $k$ is an
$R$-module.  Recall that a subcategory of the category of $R$-modules
is said to be \emph{thick} if it is closed under (de)suspensions, equivalences,
cofibration sequences, and retracts; it is \emph{localizing} if in
addition it is closed under arbitrary coproduts.

\begin{defn}\label{DDefineCellular}
  An $R$-module is \emph{finitely built from $k$} if it belongs to the
  smallest thick subcategory of $\RMod$ which contains $k$.  An
  $R$-module is \emph{built from $k$} or is \emph{$k$-cellular} if it
  belongs to the smallest localizing subcategory of $\RMod$ which
  contains $k$.  
\end{defn}

\begin{defn}\label{DefineCellEquivalence}
  A map $f:M\to N$ of $R$-modules is a \emph{$\Cellk$-equivalence} if
  it induces an equivalence $\Hom_R(k,M)\weq\Hom_R(k,N)$.
\end{defn}

It is not hard to see that a $\Cellk$-equivalence between $k$-cellular
$R$-modules
is actually an equivalence; this follows for instance
from the fact that a $\Cellk$-equivalence $M\to N$ induces an
equivalence $\Hom_R(C,M)\weq\HomR(C,N)$ for any $k$-cellular~$C$.
The main general result in this area is an
approximation theorem. A map $M'\to M$ is said to be a
\emph{$k$-cellular approximation} if $M'$ is $k$-cellular and $M'\to
M$ is a $\Cellk$-equivalence.

\begin{thm}\label{GiveCellularApprox}\cite[I.5]{rHirschhorn}
  Any $R$-module $M$ has a functorial $k$-cellular approximation
  $\Cellk(M)\to M$. A map $M'\to M$ is a $\Cellk$-equivalence if and
  only if the induced map $\Cellk(M')\to\Cellk(M)$ is an equivalence.
\end{thm}

\begin{numbered}{Constructing $\Cellk(M)$}\label{ConstructingCellk}
In general, it is difficult to give a simple formula for $\Cellk(M)$;
the usual method for constructing it involves transfinite induction.
But let $\EC=\EndR(k)$ and note that there is a commutative diagram
\[
\begin{CD}
  \HomR(k,\Cellk M)\Tensor_{\EC} k @>>> \Cellk M\\
       @V\weq VV                  @VVV \\
  \HomR(k,M)\Tensor_{\EC} k @>>> M
\end{CD}
\]
in which the horizontal maps are evaluation.  It is easy to conclude
from this diagram that the following three conditions are equivalent:
\begin{enumerate}
  \item for all $M$, $\HomR(k,M)\Tensor_{\EC} k\to M$ is a $k$-cellular approximation,
  \item for all $M$, $\Cellk M$ is \ec/ from $k$ (\ref{DefineEC}), and
  \item any $k$-cellular $R$-module is \ec/ from $k$.
\end{enumerate}
If these conditions hold, then the functor $\Cellk(\whatever)$ is easy
to describe explicitly: it is equivalent to
$\HomR(k,\whatever)\Tensor_{\EC}k$.
\end{numbered}

We will next identify certain pairs $(R,k)$ for which the
conditions of \ref{ConstructingCellk} are satisfied.
\end{NumberedSubSection}

\begin{NumberedSubSection}{Koszul complexes}\label{DKoszul}
A \emph{Koszul complex} for an $R$-module
$k$ is an $R$-module $C$ which satisfies the following three
conditions:
\begin{enumerate}
\item $C$ is finitely built from $R$,
\item $C$ is finitely built from $k$, and
\item $C$ builds $k$.
\end{enumerate}
If $R\to k$ is a map of \salgebras, a \emph{Koszul complex for $R\to
  k$} is a Koszul complex for $k$ as a left $R$-module.

\begin{prop}\label{Proxy}
  Suppose that $R$ is \ansalgebra/ and $k$ is a module over $R$ which
  admits a Koszul complex $C$.  Then the three conditions of
  \ref{ConstructingCellk} hold for $(R,k)$.
\end{prop}

\begin{proof}
  Let $\EC=\EndR(k)$. We will prove that if $M$ is any $R$-module, then the natural map
  $\lambda:\HomR(k,M)\Tensor_{\EC}k\to  M$
    is a $k$-cellular approximation.  The domain of $\lambda$ is built
    from $k$ over $R$, because $\HomR(k,M)$ is built from $\EC$ as a
    right module over $\EC$, so it will be sufficient to prove that
    $\lambda$ is a $\Cellk$-equivalence. We look for $R$-modules $A$
    with the property that the natural map
  \[   \HomR(k,M)\Tensor_{\EC}\HomR(A,k)\to \HomR(A,M)       \]
  is an equivalence. The module $A=k$ certainly works, and hence so
  does any module finitely built from $k$, e.g., the Koszul complex
  $C$. Since $C$ is finitely built from $R$,
  $\HomR(k,M)\Tensor_{\EC}\HomR(C,k)$ is equivalent to
  $\HomR(C,\HomR(k,M)\Tensor_{\EC}k)$.  The conclusion is that the map
  $\lambda$ is a $\Cell_C$-equivalence. Since $C$ builds $k$, it follows
  that the map is also a $\Cellk$-equivalence.
\end{proof}

In the presence of a Koszul complex, it is easier to recognize
Gorenstein homomorphisms.

\begin{prop}\label{RecognizeGorenstein} \cite[8.4]{rDGI}
  Suppose that $R\to k$ is a map of \salgebras/ such that $k$, as an
  $R$-module, admits a Koszul complex. Then $R\to k$ is Gorenstein if
  and only if there is some integer $d$ such that $\Hom_R(k,R)$ is
  equivalent to $\Sigma^dk$ as a module over $k$.
\end{prop}

\begin{proof}
  Since $\Cellk(R)$ is effectively constructible from $k$
  (\ref{Proxy}), the map $R\to k$ is Gorenstein if and only if there
  is some integer $d$ such that $\Hom_R(k,\Cellk(R))$ is equivalent to
  $\Sigma^dk$ as a $k$-module. The proposition follows from the fact
  that the cellular approximation map $\Cellk(R)\to R$ induces an
  equivalence on $\HomR(k,\whatever)$.
\end{proof}
\end{NumberedSubSection}

\begin{NumberedSubSection}{Examples of Koszul complexes}

  Suppose that $R$ is  an ordinary 
  commutative ring and that $k$ is a
  field which is a quotient of $R$ by a finitely generated ideal
  $\langle r_1,\ldots,r_m\rangle$.  Let $C_i$ denote the complex
  $R\RightArrow{r_i}R$ (concentrated in degrees $0$ and $-1$), and $C$
  the complex $C_1\Tensor_R\cdots\Tensor_RC_m$. This is what is
  usually called the Koszul complex for $R\to k$; the following
  shows that definition \ref{DKoszul} is consistent with this usage.

\begin{prop}\label{AlgebraicKoszul} \cite[3.2]{rDGI}
  In the above situation, $C$ is a Koszul complex for $R\to k$ (in the
  sense of \ref{DKoszul}).
\end{prop}

Recall that $\SKn$ is the localization of $\Sph$ with respect to the
Morava $K$-theory $K(n)$. The unit map $\Sph\to \Ep$ extends
uniquely to an \salgebra/ map $\SKn\to E$.

\begin{prop} \label{SisFinite}
 The spectrum $\SK$ is a Koszul complex for $\SK\to \Ep$. 
\end{prop}

\begin{proof}
  It follows from \cite[8.9, p.~48]{rHS}, that $\SK$ is finitely built
  from  from  \En/  (but  don't  ignore   the notational   discrepancy
  described in the proof of  \ref{InvertibleMeansNiceK}). It is  clear
  that  $\SK$   finitely  builds  itself,  and,   since  $\Ep$   is an
  $\SK$-module, that $\SK$ builds \Ep.
\end{proof}

  Let \Typen/ be a fixed finite complex of type~$n$ (\ref{SNotation}).

\begin{prop}\label{knIsKoszul}
  The $\Sn$-module $\Typen$ is a Koszul complex for $\Sn\to\kn$.
\end{prop}

\begin{proof}
  By construction, $\Typen$ is finitely built from \Sn. Since
  $\kn\Tensor\Typen$ is a nontrivial sum of copies of \kn, it is clear
  that \Typen/ builds \kn. Finally, \cite[8.12]{rHS} shows that \Typen/ is
  finitely built from \kn.
\end{proof}

\begin{prop}
  Let $\EndE$ denote the endomorphism spectrum $\End(\En)$. Then $\En$
  is a Koszul complex for itself as a module over $\EndE$.
\end{prop}

\begin{proof}
  As above, \SK/ is finitely built from \En.  It follows immediately
  that $\En=\Hom(\SK, \En)$ is finitely built from
  $\EndE=\Hom(\En,\En)$ as a left module over \EndE.
\end{proof}

\end{NumberedSubSection}

\begin{NumberedSubSection}{Self-dual Koszul complexes}
Suppose that $R$ is a commutative \salgebra. A module $M$ over $R$ is
said to be \emph{self-dual with respect to
Spanier-Whitehead duality} if  there is some integer $e$ such that
$\HomR(M,R)$ is equivalent to $\Sigma^eM$ as an $R$-module. The
following observation is less specialized than it seems.

\begin{prop}
  Suppose that $R$ is a commutative \salgebra, and that $k$ is an
  $R$-module which admits a Koszul complex with is self-dual with
  respect to Spanier-Whitehead duality. Then a map $f:M\to M'$ of
  $R$-modules is a $\Cellk$-equivalence if and only if it induces an
  equivalence $k\Tensor_RM\to k\Tensor_RM'$.
\end{prop}

\begin{proof}
  Let $C$ be the self-dual Koszul complex. Since $C$ and $k$ build one
  another, $f$ induces an equivalence on $\Hom_R(k,\whatever)$
  (i.e., is a $\Cellk$-equivalence)  if and only if it induces an
  equivalence on $\Hom_R(C,\whatever)$.
  Moreover, $f$ induces an equivalence on
  $k\Tensor_R\whatever$ if and only if it induces an equivalence on
  $C\Tensor_R\whatever$. Since $C$ is finitely built from $R$, the
  functor $\HomR(C,\whatever)$ is equivalent to
  $\HomR(C,R)\Tensor_R\whatever$. The proposition follows from the
  fact that $\HomR(C,R)$ is equivalent to $\Sigma^eC$.
\end{proof}

\begin{ex}
  The Koszul complex from \ref{AlgebraicKoszul} is self-dual; compare
  \cite[6.5]{rDwGr}. 
\end{ex}

\begin{ex}\label{TwoKEquivalences}
  The Koszul complex \Typen/ from \ref{knIsKoszul} can be chosen to be
  self-dual; just replace \Typen/ if necessary by
  $\Typen\Tensor_{\Sn}\Dual\Sn\Typen$.  It follows that a map of
  $\Sn$-modules is a $\Cell_{\kn}$-equivalence if and only if it induces an
  isomorphism on $\kn_*$. In particular, for any \Sn-module $X$ the
  map $X\to L_{\kn}X$ is a $\Cell_{\kn}$-equivalence and the map
  $\Cell_{\kn}X\to X$ is a $\kn_*$-equivalence.
\end{ex}
  
\end{NumberedSubSection}

\section{Commutative Rings}\label{CCommutative}

In this section we will look at several examples of Gorenstein
homo\-mor\-phisms $R\to k$ between ordinary noetherian commutative
rings. In each case $R$ is a regular ring, and $R\to k$ is projection
to a residue  field. In this situation $R\to k$ is Gorenstein
\cite{Matsumura}, i.e., 
$\Ext^i_R(k,R)$ vanishes except in one degree, and in that
degree is isomorphic to $k$ (\ref{RecognizeGorenstein},
\ref{AlgebraicKoszul}). Indeed, to see this localize $R$ if necessary at the maximal ideal
$\mfm=\ker(R\to k)$ and observe that the Koszul complex on a minimal
generating set for $\mfm$ is a resolution of~$k$. It is then clear
from calculation that $\Ext^*_R(k,R)$ is isomorphic to (a shift of)~$k$. There are three examples; the third is a
combination of the first two. We give special attention to the
extended functoriality issues discussed in \ref{LFunctoriality}. In
this section we sketch arguments which explain where the results come from;
these issues are treated in \cite{rHuang} from a very different point
of view.

\begin{NumberedSubSection}{$p$-adic number rings}\label{padicGor}
  Let $R$ be the ring $\Zp$ of $p$-adic integers, and $k$ the finite
  field $R/pR\cong\Z/p$. The ring $R$ is regular, hence $R\to k$ is
  Gorenstein  and there is a Brown-Comenetz dualizing module $\IIRG=\Cellk(R)$
  provided by the Gorenstein condition.
If $M$ is a finitely generated $p$-primary torsion abelian group, the
associated notion of duality is given by
\[
\Dual\IIRG(M)\weq\Sigma^{-1}\Ext^{1}_R(M,R)\,.
\]
The $\Ext$-group on the right is naturally isomorphic to the
Pontriagin dual of $M$, and in fact the short exact sequence
\[
  0\to \Zp\to\Qp\to \Z/p^\infty\to 0
\]
can be used to produce an equivalence
$\IIRG\weq\Sigma^{-1}\Z/p^\infty$.  All extended naturality issues
(\ref{LFunctoriality}) are trivial, if only because $R$ has no nontrivial
automorphisms. In this case Gorenstein duality and Pontriagin duality
coincide (up to suspension) on the category of \finitelength/ (\ref{FiniteLength}) skew $R$-modules.

A more interesting possibility is to let $R$ be the ring of integers
in a finite unramified extension field of \Qp, and $k$ the residue
field of $R$. Again $R$ is regular, $R\to k$
is Gorenstein, and there is a  Brown-Comenetz dualizing module
$\IIRG=\Cellk(R)$ provided by the Gorenstein condition. However, as in
\ref{InduceItUp} there is also \acoinduced/ Brown-Comenetz dualizing module,
given by $\IIRO=\Cell_k \Hom_{\Zp}(R,\Zpinfty)$.
For \adiscrete/ finitely-generated $p$-primary torsion $R$-module $M$,
the two associated notions of duality are given by
\[
     \begin{aligned}
     \Dual{\IIRG}(M)&\weq \Sigma^{-1} \Ext^1_R(M,
     R)\\
     \Dual\IIRO(M)&\weq\hphantom{\Sigma^{-1}}
     \Ext^0_{\Zp}(M,\Zpinfty)\,,
     \end{aligned}
\]
where as before the lower $\Ext$-group is the Pontriagin dual of $M$.
Perhaps surprisingly, the two $\Ext$-functors on the right are
naturally isomorphic on the category of \finitelength/ skew $R$-modules.  This can be proved by showing
that there are equivalences
\[
\begin{aligned}
  \IIRG&\weq \Sigma^{-1}\Zpinfty\Tensord_{\Zp}R\\
  \IIRO&\weq \hphantom{\Sigma^{-1}}\Ext^0_{\Zp}(R,\Z/p^\infty)
\end{aligned}
\]
and observing that there is 
a canonical isomorphism
\[
       R \to \Ext^0_{\Zp}(R,\Zp)
\]
given by the map which sends $r\in R$ to the trace over $\Zp$ of the
map $x\mapsto rx$.  These considerations produce an
$R[\Aut(R)]$-equivalence $\Sigma\IIRG\weq\IIRO$.  Hence in this case,
also, Gorenstein duality agrees up to suspension with Pontriagin
duality as strongly as we might hope.
\end{NumberedSubSection}

\begin{NumberedSubSection}{Power series over a field}
  Suppose that $R$ is the power series ring $k\lbbracket x_1,\ldots,x_{n-1}\rbbracket$, and that $R\to k$ is the natural map sending  $x_i$ to zero. The ring $R$ is regular, $R\to k$ is Gorenstein, 
  and the Gorenstein condition provides a Brown-Comenetz
  dualizing module $\IIRG=\Cell_k(R)$.  As in \ref{InduceItUp}, there is \acoinduced/ Brown-Comenetz dualizing
  module $\IIRO=\Cell_k\Ext^0_k(R, k)$. 
  Let $\mfm$ denote the kernel of $R\to k$. For \afinitelength/ (\ref{FiniteLength}) $R$-module $M$, the two
  associated notions of duality are given by
  \begin{equation}\label{PSFieldDualities}
      \begin{aligned}
      \Dual\IIRG(M)&\weq \Sigma^{-(n-1)}\Ext^{n-1}_R(M,R)\\
      \Dual\IIRO(M)&\weq\hphantom{\Sigma^{-(n-1)}}\Ext^0_k(M,k)
      \end{aligned}
  \end{equation}

  Let $\RR'$ denote $R\Tensord_kR\iso
  k\lbbracket x_1\,\ldots,x_{n-1}\rbbracket\Tensord_kk\lbbracket y_1,\ldots,y_{n-1}\rbbracket$,
  let $\RR=k\lbbracket
  x_1,\ldots,x_{n-1},y_1,\ldots,y_{n-1}\rbbracket$ be the evident
  completion of $\RR'$, and let \Line/ be given by the formula
  \[\Line=\pi_{n-1}(R\Tensor_{\RR}R)\iso\Tor_{n-1}^{\RR}(R,R)\,.\] (Here $R$ is
  treated as an $\RR$-module by the completed multiplication map
  $\RR\to R$ which has $x_i\mapsto x_i$ and $y_i\mapsto x_i$.)  Note that $\Aut(R)$ acts
  naturally on \Line. The following proposition compares the two
  dualities of \ref{PSFieldDualities}.

  \begin{prop}\label{TwistOverk}
    The object \Line/ is a free (\discrete) $R$-module on one
    generator. For any \finitelength/ $R$-module $M$, there is an
  isomorphism
  \[
     \Line\Tensord_R\Ext^{n-1}_R(M,R)\iso \Ext^0_k(M,k)\,.
  \]
   which is natural with respect to skew homomorphisms
   $M\to M'$.
  \end{prop}

  \begin{rem}
    Underlying \ref{TwistOverk} is an $R[\Aut(R)]$-equivalence
    \[\Sigma^{n-1}\Line\Tensor_R\IIRG\weq\IIRO\,.\]
    or an equivalence $\Sigma^{n-1}\Line\weq\Hom_R(\IIRG,\IIRO)$. On the indicated category of skew $R$-modules, Gorenstein duality
    agrees naturally (up to suspension) with Kronecker duality over
    $k$ only after twisting by $\Line$.
  \end{rem}

  Let $\Rt$ denote $R$ considered as \adiscrete/ $\RR$-module via the
  map $\RR\to R$ with $x_i\mapsto 0$ and $y_i\mapsto x_i$. 

  \begin{lem}\label{TwoTensors}
    The natural map $R\Tensor_{\RR'}\Rt\to R\Tensor_{\RR}\Rt$ is an
    equivalence. 
  \end{lem}

  \begin{proof}
    This follows from an explicit calculation depending on the fact
    that for both $\RR$ and $\RR'$, the module $\Rt$ is the quotient
    of the ring by the ideal generated by the regular sequence
    $(x_1,\ldots,x_{n-1})$. 
  \end{proof}

  In the following lemma, $\RR$ acts on
  $\Hom_k(M,R)\weq\Ext^0_k(M,R)$ in a completed bimodule fashion,
  e.g., $(x_i\cdot f)(m)=f(x_im)$ and $(y_i\cdot f)(m)=y_if(m)$.

  \begin{lem}\label{GetTheSameAnswer}
    If $M$ is \afinitelength/
    $R$-module, then the natural maps
    \[
    \begin{aligned}
      \Hom_{\RR}(R,\Hom_k(M,R))&\to\Hom_{\RR'}(R,\Hom_k(M,R))\\
       R\Tensor_{\RR'}\Hom_k(M,R)&\to R\Tensor_{\RR}\Hom_k(M,R)
    \end{aligned}
    \]
    are equivalences.
  \end{lem}

  \begin{proof}
  \pushQED{\qed} %
    The module $M$ has a composition series in which the successive
    quotients are isomorphic to $k$; by an inductive argument, it
    suffices to treat the case $M=k$. In this case the second
    statement is \ref{TwoTensors}, while the first follows from
    \ref{TwoTensors} and the equivalences
     \begin{align*}
     \Hom_{\RR}(R,\Rt) &\weq \Hom_R(\Rt\Tensor_{\RR}R, \Rt)\\
     \Hom_{\RR'}(R,\Rt) &\weq\Hom_R(\Rt\Tensor_{\RR'}R,\Rt)\,. \qedhere
     \end{align*} 
   \end{proof}

     We will use the fact that for any $R$-modules
    $A$ and $B$, there are natural weak equivalences
    \begin{equation}\label{TensorsAndHoms}
       \begin{aligned}
          \Hom_R(A,B)&\weq\Hom_{\RR'}(R, \Hom_k(A,B))\\
          A\Tensor_RB&\weq R\Tensor_{\RR'}(A\Tensor_kB)
       \end{aligned}
    \end{equation}

  \begin{titled}{Proof of \ref{TwistOverk} (sketch)} The fact that
    \Line/ is a free module of rank~1 over $R$ follows from 
    calculation with the usual Koszul resolution of $R$ over
    $\RR$. Let $\Lined$ denote $\Ext^{n-1}_S(R,S)$. Another
    calculation with the Koszul resolution shows that $\Lined$ is also
    a free module of rank~1 over~$R$, and that
     the composition pairing
    \[\Line\Tensord_R\Lined=
    \Tor_{n-1}^S(R,R)\Tensord_R\Ext^{n-1}_S(R,S)\to \Tor_0^S(R,S)\cong R\]
    is an isomorphism; see \cite[Lemma~1.5]{AILN}. To finish the proof, it is enough to show that
    for any~$M$ as described there is a natural isomorphism
    \[     \Lined\Tensord_R\Ext^0_k(M,k)\to \Ext^{n-1}_R(M,R)\,.\]
    Again, consideration of the Koszul resolution shows that $R$ is
    finitely built from~$S$ as an $S$-module, and that 
   $\Ext^i_S(R,S)$ vanishes if $i\ne n-1$. It follows that 
    $\Hom_S(R,S)$ is equivalent to~$\Sigma^{1-n}\Lined$, and that for
    any $S$-module $X$ there is a natural isomorphism
    \begin{equation}\label{FundamentalPairing}
      \Sigma^{1-n}\Lined\Tensor_SX\sim
      \Hom_S(R,S)\Tensor_SX\sim\Hom_S(R,X)\,.
    \end{equation}

    Now let $M$ be \afinitelength/ $R$-module, and let $X$ be the \discrete/
    $S$-module $\Hom_k(M,R)$. The module  $M$ is finite-dimensional over $k$,
    and so $X$ is equivalent to $\Hom_k(M,k)\Tensor_kR$, and (cf.
    \ref{GetTheSameAnswer}, \ref{TensorsAndHoms})
    \ref{FundamentalPairing} gives an equivalence
    \[\begin{aligned}
    \Sigma^{1-n}\Lined\Tensor_R\Hom_k(M,k)&\weq\Hom_S(R,S)\Tensor_R(R\Tensor_S(\Hom_k(M,k)\Tensor_kR))\\
             &\weq\Hom_S(R,S)\Tensor_S\Hom_k(M,R)\\
             &\weq\Hom_S(R,\Hom_k(M,R))\\
             &\weq\Hom_R(M,R)\,.
    \end{aligned}
    \]
    Applying $\pi_{1-n}$ gives the desired isomorphism. Its
    construction is natural enough to respect skew
    homormorphisms $M\to M'$. \qed
   \end{titled}
\end{NumberedSubSection}

\begin{NumberedSubSection}{Power series over a $p$-adic ring}\label{SFullDuality}
  Let $W$ be the ring of integers in a finite unramified extension
  field of $\Qp$, $k$ the residue field of $W$, $R$ the formal
  power series ring $W\lbbracket x_1,\ldots,x_{n-1}\rbbracket$, and
  $R\to k$ the quotient map sending each $x_i$ to zero. As before, $R\to k$
  is Gorenstein  and the Gorenstein
  condition provides a Brown-Comenetz dualizing module
  $\IIRG=\Cellk(R)$. As in \ref{InduceItUp}, there is \acoinduced/
  Brown-Comenetz dualizing module 
  $\IIRO=\Cellk\Hom_{\Zp}(R,\Zpinfty)$. Let \mfm/ denote the kernel of
  $R\to k$. For \afinitelength/ (\ref{FiniteLength})
  $R$-module $M$, the two associated notions of duality are given by
  \[
      \begin{aligned}
      \Dual\IIRG(M)&\weq \Sigma^{-n}\Ext^{n}_R(M,R)\\
      \Dual\IIRO(M)&\weq\hphantom{\Sigma^{-0}}\Ext^0_{\Zp}(M,\Zpinfty)
      \end{aligned}
  \]
  Let $S=W\lbbracket x_1,\ldots,x_{n-1},y_1,\ldots,y_{n-1}\rbbracket$
  be the evident completion of \hbox{$R\Tensord_WR$}, and let
  $\Line=\pi_{n-1}(R\Tensor_SR)$.

  \begin{prop}\label{TwistOverO}
    The object \Line/ is a free \discrete/ $R$-module on one
    generator. For any \finitelength/ $R$-module $M$, there is an
  isomorphism
  \[
     \Line\Tensord_R\Ext^{n}_R(M,R)\iso \Ext^0_{\Zp}(M,\Zpinfty)\,.
  \]
   which is natural with respect to skew homomorphisms
   $M\to M'$.
  \end{prop}
  
  \begin{rem}
    Behind this proposition is an equivalence
    $\Sigma^n\Line\Tensor_R\IIRG\weq\IIRO$. On the indicated category
    of skew $R$-modules, Gorenstein duality agrees naturally (up to
    suspension) with Pontriagin duality only after twisting by
    $\Line$. 
  \end{rem}

  \begin{titled}{Proof of \ref{TwistOverO} (sketch)}
    Let $\Lined$ denote $\Ext^{n-1}_S(R,S)$. As in the proof
    of \ref{TwistOverk}, it is enough to show that for all $M$ of the
    indicated type 
    there is a natural isomorphism
    \[   \Lined\Tensord_R\Ext^0_{\Zp}(M,\Zpinfty)\cong
    \Ext^{n}_R(M,R)\,.\]
     Let $\mfn\subset\mfm$ denote the kernel of the map $R\to W$ sending each $x_i$
    to zero. The arguments in the proof of \ref{TwistOverk} give an
    equivalence 
    \[
       \Lined\Tensor_R\Hom_{W}(M,W)\weq\Sigma^{n-1}\HomR(M,R)
    \]
    for any \discrete/ finitely-generated $R$-module $M$ which is
    annihilated by a power of~\mfn. (The observation
    in the proof of \ref{GetTheSameAnswer} that $M$ has a composition
    series in which the successive quotients are isomorphic to~$k$ has
    to be replaced by the observation that $M$ has a composition
    series in which the successive quotients are isomorphic as
    $R$-modules to cyclic modules over the PID~$W$.)
     If  in addition $M$ is
    \mfm-primary, i.e., if $M$ is a $p$-primary torsion abelian group,
    then  the considerations of  \ref{padicGor} give an equivalence
    \[
        \Sigma\Hom_W(M,W)\weq\Hom_{\Zp}(M,\Zpinfty)\,.
    \]
    Combining the equivalences, applying $\pi_*$, and verifying
    naturality gives the result. \qed 
  \end{titled}
 \end{NumberedSubSection}

\section{Chromatic ingredients}\label{CPreliminaries}

The purpose of this section is to recall some material from \cite{rHS}
and \cite{rS}.  As in \ref{GHIntro}, let \IIR/ denote
$\Cell^{\Sn}_{\kn}\Hom(\Sn,\II)$, where \II/ is the ordinary
Brown-Comenetz dual of the sphere.

\begin{rem}\label{IIRGivesPontriagin}
  Note that if $X$ is an \Sn-module which is built from \kn, then
  $\Hom(X,\IIR)\weq\Hom(X,\Hom(\Sn,\II))$ is equivalent to
  $\Hom(\Sn\Tensor X,\II)\weq\Hom(X,\II)$. In particular, for
  such an $X$ the homotopy groups of $\Dual\IIR X$ are the Pontriagin
  duals of the homotopy groups of $X$.
\end{rem}

\begin{prop}\label{KnIsBC}
  The \Sn-module \IIR/ is a Brown-Comenetz dualizing module for
  $\Sn\to\kn$. 
\end{prop}

\begin{proof}
  Given \ref{IIRGivesPontriagin}, a homotopy group calculation
  shows that $\Hom(\kn,\II)$ is equivalent to $\kn$ as a left
  $\kn$-module.  Since $\Sn\to \kn$ has a Koszul complex
  (\ref{knIsKoszul}), the result follows from
  \ref{RecognizeGorenstein}.
\end{proof}

Recall that a
\kn-local spectrum $X$ is said to be \emph{invertible} if there exists
a \kn-local spectrum $Y$ such that $X\KSmash Y\weq\SKn$. In the
following statement ``shifted isomorphic'' means ``isomorphic up to
suspension''. 

\begin{prop} \label{InvertibleMeansNiceK}
  Suppose that $X$ is a \kn-local spectrum. Then the following 
  conditions are equivalent:
  \begin{enumerate}
  \item $X$ is invertible.
  \item $\kn_* X$ is shifted isomorphic to $\kn_*$ as a $\kn_*$-module.
  \item $\kn^*X$ is shifted isomorphic to $\kn^*$ as a $\kn^*$-module.
  \item $\Eph_* X$ is shifted isomorphic to $\Ep_*$ as an $\Ep_*$-module.
    \label{Eversion}
  \item $\Ep^*X$ is shifted isomorphic to $\Ep^*$ as an
    $\Ep^*$-module. \label{Ecversion}
  \end{enumerate}
\end{prop}

\begin{proof}
  This is essentially \cite[14.2]{rHS}. There is a technical point
  to take into account. Hovey and Strickland use the letter ``$E$'' to
  denote a spectrum which we will call $\EHS$; its homotopy groups are
  given by
  \[
      \EHS_* = \Zp\lbbracket
    v_1,v_2,\ldots,v_{n-1}\rbbracket[v_n,v_n^{-1}]\,.
  \] 
  where $|v_k|=2(p^k-1)$.
  Our ring $\Ep_*$ is a finitely generated free module over $\EHS_*$
  under the map sending $v_k$ to $u^{p^{k-1}}u_k$, where $u_0=p$,
  $u_n=1$. Let $\EHS^\vee_*(X)$ denote $\pi_*L_{\kn}(\EHS\Tensor X)$. 
  Given the way in which the cohomology theories $\EHS_*$ and $\Ep_*$
  are defined (i.e., by Landweber exactness \cite[p.~4]{rHS}
  \cite{rS}), for any
  spectrum $X$ there are isomorphisms
  \begin{equation}\label{BigVsSmall}
    \begin{aligned}
      \Ep_*(X)&\iso\Ep_*\Tensord_{\EHS_*}\EHS_*(X)\\
      \Eph_*(X)&\iso\Ep_*\Tensord_{\EHS_*}\EHS_*^\vee(X)\,.
    \end{aligned}
  \end{equation}
  Hovey and Strickland show that conditions (1) and (2) and (3) of the
  proposition hold if and only if $\EHS^\vee_*(X)$ is isomorphic to
  $\EHS_*$ (up to suspension).  The proof is completed by observing that,
  in view of \ref{BigVsSmall}, $\Eph_*(X)$ is isomorphic 
  to $\Ep_*$ (up to suspension) if and only if $\EHS^\vee_*(X)$
  is equivalent to $\EHS_*$ (up to suspension). Similar considerations
  apply to $\Ep^*$.
\end{proof}

\begin{prop} \label{IisInvertible}
  The \kn-local spectrum \IIRhat/ is invertible. 
\end{prop}

\begin{proof}
  This is \cite[10.2(e)]{rHS}; see also
  Theorem~\ref{BCequalsInvertible}. 
\end{proof}

\begin{prop}\label{SelfEquivalence}
  If $I$ is an invertible \kn-local spectrum, then the functor
  $X\mapsto X\KSmash I$ gives a self-equivalence of the homotopy
  category of \kn-local spectra. In particular, for any \kn-local
  spectra $X$, $Y$ the natural map
  \[
    \Hom(X,Y)\to\Hom(X\KSmash I, Y\KSmash I)
  \]
  is an equivalence.
\end{prop}

\begin{proof}
  The inverse functor is given by $X\mapsto X\KTensor J$, where
  $I\KTensor J\weq\SK$.
\end{proof}

\begin{rem}\label{DualOfInvertible}
  If $I$ is invertible, the ``multiplicative inverse''
  $J$ of $I$ is given by $J=\Hom(I,\SK)$. This can be derived from the
  chain of equivalences
  \[
    J\weq\Hom(\SK,J)\weq \Hom(I\KTensor\SK,I\KTensor J)\weq
    \Hom(I,\SK)\,.
  \]
\end{rem}

\begin{prop}\label{InvertibleMeansDualizable}\cite[10.6]{rHS}
  If $I$ is an invertible \kn-local spectrum, then for any spectrum
  $X$, the natural map $\Hom(X,\SK)\KSmash
  I\to\Hom(X, I)$ is an equivalence.
\end{prop}

\begin{proof}
  Pick a \kn-local $J$ such that $I\KSmash J\weq\SK$. 
  Now use \ref{SelfEquivalence} to compute
  \[
    \begin{aligned}
    \Hom(X,I)&\weq\Hom(X\KSmash J,I\KSmash J)\\
    &\weq\Hom(J,\Hom(X,\SK))\\
    &\weq\Hom(J\KSmash I,\Hom(X,\SK)\KSmash I)\,
   \end{aligned}
  \]
  and note that the final spectrum is $\Hom(X,\SK)\KSmash I$.
\end{proof}

\begin{thm}\label{EtoSKisGor}\label{DualOfE}
  \cite[Prop.~16]{rS} 
  There is a weak equivalence
  \begin{equation}
            \DK\Ep\weq \Sigma^{-n^2}\Ep
  \end{equation}
  of left \En-modules, which respects the actions of $\Gamma$ on both
  sides.
\end{thm}

\begin{proof}
   Much of the content of this proof is in the technical details, but
  we will sketch the argument. Let $\EndE=\End(\EPP)$. Note that
the natural map
 \begin{equation}\label{ESpectralSeq}
     \Hom(\EPP,X)\to \Hom_{\EndE}(\Hom(X,\EPP), \Hom(\EPP,\EPP))
  \end{equation}
  is a weak equivalence for $X=\EPP$.  Since $\SK$ is finitely built
  from \En/ (\ref{SisFinite}) and both sides of \ref{ESpectralSeq}
  respect cofibration sequences in $X$, it follows that
  \ref{ESpectralSeq} is an equivalence for $X=\SK$. This results in a
  strongly convergent
  Adams spectral sequence
  \[
  E^2_{*,*}=\Ext^*_{\EPP_*\lbbracket\Gamma\rbbracket}
  (\EPP_*,\EPP_*\lbbracket\Gamma\rbbracket)\Rightarrow\pi_*\Hom(\EPP,\SK)\,.
  \] 
  By a change of rings, the $E^2$-page is isomorphic to the continuous
  cohomology $H_c^*(\Gamma,\EPP_*\lbbracket\Gamma\rbbracket))$.  Since
  $\Gamma$ is a Poincar\'e duality group of dimension $n^2$, this
  continuous cohomology vanishes except in homological
  degree~$n^2$, where it is isomorphic to $\Ep_*$
  \cite[Prop.~5]{rS}.
\end{proof}

\begin{rem}\label{EDoubleDual}
  It follows from \ref{DualOfE} that the natural map
  $\rho:\EPP\to \DK^2\EPP$ is an equivalence.
  To see this, let $f:\Sigma^{-n^2}\EPP\to\SK$ be a map
  which corresponds under \ref{DualOfE} to the unit in $\EPP_0$. The
  adjoint of the equivalence $\Sigma^{-n^2}\EPP\to\Hom(\EPP,\SK)$ is then
  the composite
  \begin{equation}\label{FirstAdjoint}
       (\Sigma^{-n^2}\EPP)\Tensor \EPP
      \RightArrow{m}\Sigma^{-n^2}\EPP\RightArrow{f}\SK\,,
  \end{equation}
  where $m$ is obtained from the multiplication on $\EPP$. 
 Consider
  the following two maps
  \[\rho,\lambda:\EPP\to\DK^2\EPP\weq\Hom(\Sigma^{-n^2}\EPP,\SK)\,,\]
  which we will specify by giving their adjoints
  $\EPP\Tensor\Sigma^{-n^2}\EPP\to\SK$.  The adjoint of $\rho$ is the
  composite of \ref{FirstAdjoint} with the transposition map
  $\EPP\Tensor(\Sigma^{-n^2}\EPP)\to (\Sigma^{-n^2}\EPP)\Tensor \EPP$;
  the adjoint of $\lambda$ is obtained by shifting the suspension
  coordinate in \ref{FirstAdjoint} from one tensor factor to the
  other. The map $\lambda$ is an equivalence because the adjoint of
  \ref{FirstAdjoint} is an equivalence.  The fact that $\rho$ is an
  equivalence now follows from the fact that $\EPP$ is a commutative
  \salgebra/.
\end{rem}

\section{Gross-Hopkins duality}\label{SGrossHopkins}

In this section we prove the main statements involved in Gross-Hopkins
duality, except, of course, for the Gross-Hopkins calculation itself
(\ref{GHTheorem}). We rely heavily on \cite{rHS} and \cite{rS}.

  \begin{titled}{Proof of \ref{GHShift}}
    By \ref{TwoKEquivalences}, $\Hom(\kn,\Sn)$ is equivalent to
    $\Hom(\kn,\SK)$.  Use \ref{IisInvertible} and
    \ref{SelfEquivalence} to obtain an equivalence
    $\Hom(\kn,\SK)\weq\Hom(\kn\KSmash\IIRhat,\IIRhat)$,  observe
    (\ref{InvertibleMeansNiceK}) that $\kn\KSmash\IIRhat$ is
    equivalent to \kn, and invoke \ref{KnIsBC} to evaluate
    $\Hom(\kn,\IIRhat)\weq\Hom(\kn,\IIR)$.
  \end{titled}

  \begin{titled}{Proof of \ref{GHSame}}
    For the statment involving \IIRG, note that the map $\IIRG\to\Sn$
    is a $\Cell_{\kn}$-equivalence, and hence
    (\ref{TwoKEquivalences}) an equivalence on $\kn_*$ or $\Eph_*$. It
    follows that $\Eph_*\IIRG$ is isomorphic to $\Eph_*\Sn\iso\Ep_*$,
    even as modules over $\Gamma$.  The statement involving $\IIR$ is
    a consequence of \ref{IisInvertible} and
    \ref{InvertibleMeansNiceK}, since the localization map
    $\IIR\to\IIRhat$ induces an isomorphism on $\kn_*$ or $\Eph_*$.
    \qed
  \end{titled}
  
  \begin{titled}{Proof of \ref{SimpleNaturality}}
    For the first isomorphism, observe that because $\Typen$ is
    finitely built from \kn/ \cite[8.12]{rHS} and $\IIRG\to\Sn$ is a
    $\Cell_{\kn}$-equivalence, $\Dual\IIRG(\Typen)$ is equivalent to
    $\Dual\Sn(\Typen)$. Since $\Typen$ is finitely built from \Sn, the
    usual properties of Spanier-Whitehead duality give an equivalence
  \[
        \Ep\Tensor\Dual\Sn(\Typen)\weq\Hom(\Typen,\Ep)\,.
  \]
  It follows from \ref{GHShift} that $\Dual\Sn(\Typen)$ is also
  finitely built from \kn, which implies that
  $\Ep\Tensor\Dual\Sn(\Typen)$ is \kn/-local and hence equivalent to
  $\Ep\KTensor\Dual\Sn(\Typen)$. Combining these obervations gives an
  equivalence
  \[
  \Ep\KTensor\Dual\IIRG(\Typen)\weq\Hom(\Typen,\Ep)\,,
  \]
  so that $\Eph_i\Dual\IIRG(\Typen)$ is isomorphic to $\Ep^{-i}(\Typen)$.
  There is a strongly convergent universal coefficient spectral
  sequence
  \[
     \Ext^*_{\Ep_*}(\Ep_*\Typen,\Ep_*)\Rightarrow \Ep^*(\Typen)\,.
  \]
  Since $\Ep_*$ is isomorphic as a graded $\Ep_*$-module
  to $\Ext^0_{\Ep_0}(\Ep_*,\Ep_0)$, a standard change of rings
  argument (Shapiro's lemma) produces a spectral sequence
  \[
  \Ext^i_{\Ep_0}(\Ep_j\Typen,
  \Ep_0)\Rightarrow\Ep_{-j-i}\Dual\IIRG(\Typen)\,.
  \]
  But $\Ep_0\to\Fpn$ is Gorenstein and each group $\Ep_j\Typen$ has a
  finite composition series in which the successive quotients are
  isomorphic, as $\Ep_0$-modules, to $\Fpn$. This implies that the
  above $\Ext$-groups vanish except for $i=n$, which leads to the
  desired result.

  For the second isomorphism, observe that that there are equivalences
   \begin{equation}\label{DualityChain}
   \begin{aligned}
      \Hom(\Ep\KSmash \Typen, \IIRhat)&\weq
      \Hom(\Typen,\Hom(\Ep,\IIRhat))\\
       &\weq \Hom(\Typen,\Sigma^{-n^2}\Ep\KSmash\IIRhat)\\
       &\weq\Sigma^{-n^2}\Ep\KSmash\Hom(\Typen,\IIRhat)     
   \end{aligned}
   \end{equation}
   where the second equivalence comes from combining \ref{EtoSKisGor}
   with \ref{InvertibleMeansDualizable}, and the third from the fact
   that $\Typen$ is finite. Since $\Typen$ is finitely built out of
   \kn, $\Ep\KSmash\Typen\weq\Ep\Tensor\Typen$ is built out of \kn, and
   the homotopy groups of the initial spectrum in the chain
   \ref{DualityChain} are the Pontriagin duals of $\Eph_*\Typen$
   (\ref{IIRGivesPontriagin}). The proof is completed by noting that
   the homotopy groups of the terminal spectrum in \ref{DualityChain}
   are given by $\Eph_*\Dual\IIR \Typen$. \qed
  \end{titled}

  \begin{titled}{Proof of \ref{TwistedDifference}}
    It follows from \ref{IisInvertible},
    \ref{InvertibleMeansDualizable} and the argument in the proof of
    \ref{SimpleNaturality} that for any finite complex $\Typen$ of
    type~$n$ there is are equivalences
    $\Dual\IIRG(\Typen)\weq\Dual\Sn(\Typen)\Tensor\IIRG$ and
    $\Dual\IIR(\Typen)\weq\Dual\Sn(\Typen)\Tensor\IIR$. This gives
    Kunneth isomorphisms of modules over $\Ep\lbbracket\Gamma\rbbracket$:
  \[
       \begin{aligned}
       \Eph_*\Dual\IIR(\Typen)&\iso\Eph_*\Dual\Sn(\Typen)\Tensord_{\Ep_*}\Eph_*\IIRG\\
       \Eph_*\Dual\IIR(\Typen)&\iso\Eph_*\Dual\Sn(\Typen)\Tensord_{\Ep_*}\Eph_*\IIR
       \end{aligned}
  \]
    Let \EHS/ be the spectrum described in the proof of
  \ref{InvertibleMeansNiceK}. Call an ideal $J
  \subset\EHS_*$
  \emph{admissible} if it has the form
  $(p^{a_0},v_1^{a_1},\ldots,v_{n-1}^{a_{n-1}})$. As described in
  \cite[\S4]{rHS}, there exists a family $\{J_\alpha\}$ of
  admissible ideals, such that $\cap_kJ_\alpha=0$, and such that for
  each $\alpha$ there exists a finite complex $\Typen_\alpha$ of type~$n$
  with $\EHS_*\Typen_\alpha\iso\EHS_*/J_\alpha$. Under the inclusion
  $\EHS_*\to\Ep_*$ we can treat $J_\alpha$ as an ideal of $\Ep_*$ and
  obtain (\ref{BigVsSmall}) $\Eph_*Y_\alpha\iso\Ep_*/J_\alpha$.
  Let $X_\alpha=\Dual{\Sn}\Typen_\alpha$, so that
  $\Typen_\alpha\weq\Dual{\Sn}X_\alpha$.
  Then there are isomorphisms
  \[
       \begin{aligned}
       \Eph\Dual\IIRG X_\alpha&\iso \Eph_*(\IIRG)/J_\alpha\\
       \Eph\Dual\IIR X_\alpha&\iso\Eph_*(\IIR)/J_\alpha\,.
       \end{aligned}
  \]
  The proof is completed by combining these isomorphisms with
  \ref{SimpleNaturality} and passing to the limit in $J_\alpha$
  \cite[4.22]{rHS}.
  \end{titled}

\section{Invertible modules}

In this we prove \ref{GHClassifyBC}. We begin with an extension of \ref{InvertibleMeansNiceK}.

\begin{thm}\label{BCequalsInvertible}
  A \KPP-local spectrum $I$ is invertible if and only if, up to suspension,
 $\Hom(\kn, I)$ is equivalent to $\kn$.
  \end{thm}

\begin{rem}
  It is easy to see that $\Hom(\kn,I)$ is equivalent to $\Sigma^d\kn$
  as a spectrum if and only if it is equivalent to $\Sigma^d\kn$ as a
  \kn-module.  
\end{rem}

\begin{lem}\label{LocalMappingSpaces}
  If $Y$ is \kn-local and $X$ is any spectrum, then $\Hom(X,Y)$ is
  \kn-local.
\end{lem}

\begin{proof}
  One needs to show that if $A$ is \kn-acyclic,
  $\Hom(A,\Hom(X,Y))$ is contractible. But this spectrum is
  equivalent to $\Hom(X,\Hom(A,Y))$, and $\Hom(A,Y)$ is contractible
  because $Y$ is \kn-local.
\end{proof}

\begin{lem}\label{BCDoubleDual}
  Suppose that $I$ is a \kn-local spectrum such
  that $\Hom(\kn,I)$ is equivalent to a suspension of \kn.
  Then the natural map $\kappa_X: X\to\Dual I^2(X)$ is an equivalence for
  $X=\kn$ and $X=\SK$.
\end{lem}

\begin{proof}
  We can shift $I$ by a suspension and assume $\Hom(\kn,I)\weq \kn$.
  Let $f:\kn\to I$ be essential. Under the identification
  $\kn\weq\Hom(\kn,I)$ obtained by choosing $f$ as a generator for
  $\pi_*\Hom(\kn,I)$ as a module over $\kn_*$, the map $\kappa_{\kn}$
  is adjoint to the composite of $f$ with the multiplication map
  $\kn\Tensor\kn\to \kn$. Since $\Hom(\kn,I)$ is clearly equivalent to
  $\kn$ both as a left module and as a right module over \kn, it is
  easy to conclude that $\kappa_{\kn}$ is an equivalence
  (cf. \ref{EDoubleDual}).

  By a thick subcategory argument, $\kappa_X$ is an equivalence for
  all spectra finitely built from \kn, e.g., for a finite spectrum
  $\Typen$ of type~$n$.  Since $\Dual
  I(\Typen)\weq\Dual\Sn(\Typen)\Tensor I$ and $\Typen\weq\Dual\Sn
  ^2(\Typen)$, the spectrum $\Dual I^2(\Typen)$ can be identified with
  $\Typen\Tensor\Hom(I,I)$. It follows that
  \[
  \kn_*(\Typen)\iso\kn_*(\Typen)\Tensord_{\kn_*}\kn_*\Hom(I,I)
  \] 
  and hence that $\kn_*\Hom(I,I)\iso\kn_*$. The spectrum $\Hom(I,I)$ is
  \kn-local (\ref{LocalMappingSpaces}), is not contractible, and is an
  \salgebra/ under composition; it follows that the unit map
  $\Sphere\to\Hom(I,I)$ is nontrivial on $\kn_*$. Visibly, then, the
  unit map is an isomorphism on $\kn_*$ and induces an equivalence
  $\SK\to\Hom(I,I)$.  It is not hard to identify this equivalence with
  the natural map $\SK\to\Dual I^2(\SK)$ and conclude that
  $\kappa_{\SK}$ is an equivalence.
\end{proof}

\begin{titled}{Proof of \ref{BCequalsInvertible}}
  Suppose that $I$ is invertible.  Use \ref{InvertibleMeansDualizable}
  to deduce
  \[
     \Hom(\kn, I)\weq\Hom(\kn,\SK)\KTensor I
  \]
  and observe that both $\Hom(\kn, \SK)$ (\ref{GHShift}) and
  $\kn\KTensor I$ (\ref{InvertibleMeansNiceK}) are equivalent to \kn/
  up to suspension. The conclusion is that $\Hom(\kn,I)$ is equivalent
  to \kn/ up to suspension.

  Suppose on the other hand that $\Hom(\kn,I)$ is equivalent to \kn,
  up to suspension. It follows from \ref{BCDoubleDual} that the
  natural map
  \[
    \Hom(\kn,\SK)\sim\Hom(\kn,\Dual I^2\SK)\to\Hom(\Dual I \SK,\Dual I\kn)\weq\Hom(I,\Dual I\kn)
  \]
  is an equivalence. The conclusion is that $\kn^*I$ is isomorphic to
  $\kn_*$, up to suspension, and hence by \ref{InvertibleMeansNiceK}
  that $I$ is invertible.
\end{titled}

For the rest of this section, \EndE/ will denote the endomorphism
\salgebra/ $\End(\Ep)$ of \Ep. The left action of \Ep/ on itself
gives a ring map $\Ep\to\EndE$.

\begin{prop}\label{StrangeIsFinite}
  Suppose that $\EPP'$ is any right \EndE/-module which is equivalent
  as an \EPP-module to \EPP. Then $\EPP'$ is finitely built from
  $\EndE$ as a right module over $\EndE$.
\end{prop}

\begin{proof}
  Consider two right actions $\EPP(1)$ and $\EPP(2)$ of \EndE/ on
  \EPP/ which extend the right action of $\EPP$ on itself.  Since
  \EndE/ is in fact the endomorphism \salgebra/ of $\EPP=\EPP(1)$, the
  right action of \EndE/ on $\EPP(2)$ is determined by \ansalgebra/
  homomorphism $\alpha:\EndE\to\EndE$. For any right \EndE-module $M$,
  let $M^\alpha$ denote the right \EndE-module obtained by twisting
  the action of \EndE/ on $M$ by $\alpha$, so that
  $\EPP(2)=\EPP(1)^\alpha$. As in \cite[\S7]{rPicard},
  the homomorphism $\pi_*(\alpha):
  \EPP_*\lbbracket\Gamma\rbbracket \to
  \EPP_*\lbbracket\Gamma\rbbracket$ is determined by a cocycle
  representing an element of $H^1(\Gamma;\EPP_0^\times)$, and in
  particular, $\pi_*(\alpha)$ is an isomorphism. It follows that if
  $M$ is a free right $\EndE$-module, so is $M^\alpha$; if $M$ is
  finitely built from \EndE/ as a module over \EndE, so is $M^\alpha$.
  It suffices then to find a single example of a suitable $\EPP(1)$
  which \emph{is} finitely built from $\EndE$. For this, take
  $\EPP(1)=\Sigma^{n^2}\DK\EPP$; the distinction between the left
  action of \EPP/ on $\DK\EPP$ (\ref{DualOfE}) and the corresponding
  right action is immaterial, since \EPP/ is a commutative \salgebra.
  Since \SK/ is finitely built from \EPP/ (\ref{SisFinite}),
  $\Hom(\EPP,\SK)=\DK\EPP$ is finitely built from
  $\Hom(\EPP,\EPP)=\EndE$ as a right module over \EndE.
\end{proof}

\begin{thm}\label{InvertibleBijection}
  The functor $I\mapsto\Hom(\EPP,I)$ gives a bijection between
  equivalence classes of invertible \KPP-local spectra and equivalence
  classes of right \EndE-modules which are equivalent to \EPP/, up to
  suspension, as right \EPP-modules.
\end{thm}

\begin{rem}
    The inverse bijection sends a right module $\EPP'$ of the
    indicated type to $\EPP'\Tensor_{\EndE}E$.
\end{rem}

\begin{proof}
  First observe that if $I$ is an invertible \kn-local spectrum,
  then $\Hom(\EPP,I)$ is equivalent to \Ep/ as a right \Ep-module:
  this follows from \ref{InvertibleMeansNiceK}, together with the fact
  (\ref{InvertibleMeansDualizable}, \ref{EtoSKisGor}) that there are
  equivalences 
  \[
     \Hom(\Ep, I) \weq\Hom(\Ep,\SK)\KSmash
     I\weq\Sigma^{-n^2}\Ep\KSmash I\,.
  \]
  Next, we claim that for \emph{any} \SK-module $X$, in particular for
  $X=I$, the natural map
  \[
     \Hom(\EPP,X)\Tensor_{\EndE}\EPP\to X
  \]
  is an equivalence.  To see this, fix $X$, and consider the class of
  all spectra $Y$ such that the natural map
  \begin{equation}\label{AComparison}
     \Hom(\EPP,X)\Tensor_{\EndE}\Hom(Y,\EPP)\to\Hom(Y,X)
  \end{equation}
  is an equivalence. This class certainly includes $Y=\EPP$. Since both
  sides of \ref{AComparison} respect cofibration sequences, and $\EPP$
  finitely builds \SK/ \cite[8.9, p.~48]{rHS}, the class includes
  $Y=\SK$, which gives the desired result (cf. \cite[2.10]{rDGI}).

  Now suppose that $M$ is a right \EndE-module which is equivalent to
  \Ep/ as a right \Ep-module. Let $Y=M\Tensor_{\EndE}\Ep$. We will
  show that $Y$ is invertible, and that the natural map
  \[
        M\weq M\Tensor_{\EndE}\Hom(\Ep,\Ep)\to\Hom(\Ep,M\Tensor_{\EndE}E)=\Hom(\Ep,Y)
  \]
  is an equivalence. For the second statement,
  consider the class of right \EndE-modules $X$ with the property that
  the natural map
  \begin{equation}\label{AnotherComparison}
     X\weq
     X\Tensor_{\EndE}\Hom(\EPP,\EPP)\to\Hom(\EPP,X\Tensor_{\EndE}E)
  \end{equation}
  is an equivalence. The class certainly includes the free module $X=\EndE$,
  and hence, by a thick subcategory argument, all modules finitely
  built from \EndE.  By \ref{StrangeIsFinite}, $M$ is finitely built
  from $\EndE$, and so the class includes $M$.  Again because
  $M$ is finitely built from $\EndE$, $Y$ is finitely built from
  $\EndE\Tensor_{\EndE}E\weq E$, and so (\ref{EDoubleDual}) the natural
  maps $\Ep\to\DK^2\Ep$ and $Y\to\DK^2Y$ are equivalences.  This gives
  an equivalence
  \[
    M\weq \Hom(\Ep,Y)\weq \Hom(\DK Y, \DK\Ep)\weq\Hom(\DK
    Y,\Sigma^{-n^2}\Ep)
   \,,
  \]
  where the last equivalence is from \ref{DualOfE}.
  By \ref{InvertibleMeansNiceK}(\ref{Ecversion}), $\DK Y$ is
  invertible, and so $Y=\DK(\DK Y)$ is also invertible
  (\ref{DualOfInvertible}).
\end{proof}

\begin{titled}{Proof of \ref{GHClassifyBC}}
  By \ref{RecognizeGorenstein} and \ref{knIsKoszul}, a spectrum $X$
  which is built from \kn/ is a Brown-Comenetz dualizing module for
  $\Sn\to\kn$ if and only if $\Hom(\kn, X)$ is equivalent up to
  suspension to~\kn. It then follows from \ref{TwoKEquivalences} and
  \ref{BCequalsInvertible} that the assignment $X\to\hat X$ gives a
  bijection between equivalence classes of such Brown-Comenetz
  dualizing modules and invertible \kn-local spectra; the inverse
  bijection sends $Y$ to $\Cellk Y$. The proof is completed by invoking
  \ref{InvertibleBijection} \qed
\end{titled}

\begin{rem}
  One could consider the \emph{moduli space} $\ModPic$ of invertible
  \kn-local spectra; this is the nerve of the category whose objects
  are the invertible \kn-local spectra and whose morphisms are the equivalences
  between them \cite{rModuli}. Up to homotopy \ModPic/ can be
  identified as a disjoint union $\coprod_\alpha B\Aut(I_\alpha)$,
  where $I_\alpha$ runs through the equivalence classes of invertible
  modules, and $\Aut(I_\alpha)$ is the group-like simplicial monoid of
  self-equivalences of $I_\alpha$.  The space $\ModPic$ is an
  associative monoid, even an infinite loop space, under a product
  induced by $\KTensor$; its group of components is the Picard group
  considered in \cite{rPicard}. Let $\EPP^\times$ denote the group of
  units of the ring spectrum \EPP, so that
  $\pi_0\EPP^\times\iso\EPP_0^\times$ and
  $\pi_i\EPP^\times\iso\pi_i\EPP$ for $i>0$. It seems that one can
  construct a second quadrant homotopy spectral sequence
  \[
  E^2_{-i,j}=H_c^i(\Gamma,\pi_jB\EPP^\times)\Rightarrow
  \pi_{j-i}\ModPic
  \]   
  which above total degree~1 agrees up to a shift with the Adams spectral sequence for
  $\pi_*\SK$ (compare the proof of \ref{DualOfE}). This agreement is
  not surprising, since each component of \ModPic/ is
  $B\Sn^\times$. The edge homomorphism $\pi_0\ModPic\to
  H_c^1(\Gamma,\EPP_0^\times)$ is the map used to detect Picard group
  elements in \cite{rPicard}. The obstructions mentioned in
  \cite[p.~69]{rHS} seem to be related to the first $k$-invariant of
  $B\ModPic$ (for associative pairings) or to the first $k$-invariant
  of the spectrum obtained by delooping \ModPic/ (for associative and
  commutative pairings).

\end{rem}

\providecommand{\bysame}{\leavevmode\hbox to3em{\hrulefill}\thinspace}

\end{document}